\documentclass{amsart}
\newtheorem{proposition}{Proposition}[section]
\newtheorem{theorem}{Theorem}[section]
\newtheorem{lemma}[theorem]{Lemma}
\newtheorem{corollary}{Corollary}[section]
\theoremstyle{definition}
\newtheorem{definition}[theorem]{Definition}
\newtheorem{example}[theorem]{Example}

\theoremstyle{remark}
\newtheorem{remark}[theorem]{Remark}

\numberwithin{equation}{section}

\usepackage{xcolor}

\begin{document}

\title[Topological structure of projective Hilbert spaces]
{Topological structure of projective Hilbert spaces associated with phase retrieval vectors}

\author[F. Arabyani]{Fahimeh Arabyani Neyshaburi}
\address{Department of Mathematics and Computer Sciences, Hakim Sabzevari University, Sabzevar, Iran.}
\email{f.arabyani@hsu.ac.ir, fahimeh.arabyani@gmail.com}

\author[A. Arefijamaal]{Ali Akbar Arefijamaal}
\address{Department of Mathematics and Computer Sciences, Hakim Sabzevari University, Sabzevar, Iran.}
\email{arefijamaal@hsu.ac.ir;arefijamaal@gmail.com}
\author[Gh. Sadeghi]{Ghadir Sadeghi}
\address{Department of Mathematics and Computer Sciences, Hakim Sabzevari University, Sabzevar, Iran.}
\email{g.sadeghi@hsu.ac.ir}





\begin{abstract}
In this paper,  we explore the  interplay between topological structures and phase retrieval in the context of projective Hilbert spaces. This work provides not only  a deeper understanding and a new classification of the phase retrieval property in Hilbert spaces but also a way for further investigations into the topological underpinnings of quantum states.
\end{abstract}

\maketitle


\par

\section{Introduction}
 The framework of quantum mechanics can
be interpreted  as a  classical probability theory on the projective Hilbert
spaces \cite{Beltrametti, Busch, Holevo, qua3}. Hence,  projective Hilbert spaces, as an extension of Hilbert spaces,  are    significant specially in quantum theory. In fact,
 in quantum mechanics, the states of a quantum mechanic system  are corresponded to wave functions in the Hilbert spaces.
Furthermore,  the wave functions
$\phi$  and
$\lambda \phi$ represent the same state,  for any
$\lambda \neq 0$, and so are regarded as equivalent states. The projective Hilbert space is given by using projectivization of a Hilbert space.
More precisely, suppose $\mathcal{H}$ is a non-trivial  separable Hilbert space and  $\mathcal{H}^{*}=\mathcal{H}\setminus \{0\}$. Then the quotient space,  containing the equivalent classes   on  $\mathcal{H}^{*}$
by identifying two non-zero vectors which   differ by a complex factor,
 is known as the  projective Hilbert space in the literature of mathematics and quantum mechanics, and is usually denoted  by the symbol $P(\mathcal{H})$.
The topology and measurable structures of  the projective Hilbert spaces   have been  studied in depth, see  \cite{qua4, qua2, qua1, qua5, qua3}.
On the other hand, in the setting of separable Hilbert spaces
 a quotient space $\hat{\mathcal{H}}$  is   defined with regard to the phase retrieval theory and is closely related to the projective Hilbert space.  Actually, the problem of signal recovery without using  phase is   one of the  longstanding complications in  the context of  engineering  sciences,  including    signal and image processing, speech recognition systems, quantum theory and X-ray crystallography, where the phase information of the signal is lost. Recently, it   has been   a popular topic in mathematics and   frame theory as well  \cite{Alharbi,AAK1, AAK2, Balan-zhu, Casazza, cahill, Efriam,Van, Wang}. To describe the problem,
consider  a Bessel family of vectors
$\Phi:=\{\phi_{i}\}_{i\in I}$ in   $\mathcal{H}$   and the non-linear mapping
\begin{eqnarray}\label{Phase map}
\alpha_{\Phi}: \mathcal{H}\rightarrow l^{2}(I), \qquad \alpha_{\Phi}(x)=\lbrace \vert\langle x,\phi_{i}\rangle\vert\rbrace_{i\in I}.
\end{eqnarray}
Let us denote by $\hat{\mathcal{H}}=\mathcal{H}/\sim$
 the set of  all  equivalent classes $\hat{x}$ of vectors $x\in \mathcal{H}$  given by
\begin{equation}\label{simrel}
 \hat{x}=\{y\in \mathcal{H}; \quad y=\lambda x;  \quad     \textrm{ for some }    \lambda  \textrm{ with }  \vert \lambda \vert=1 \}.
\end{equation}
Obviously in a real Hilbert space  we have  $\hat{\mathcal{H}}=\mathcal{H}/\{1,-1\}$ and in the complex case $\hat{\mathcal{H}}=\mathcal{H}/\mathbb{T}$, where $\mathbb{T}$ is the complex unit circle. So, the mapping $\alpha_{\Phi}$ can be defined  to $\hat{\mathcal{H}}$ as $\alpha_{\Phi}(\hat{x})=\lbrace \vert\langle x,\phi_{i}\rangle\vert\rbrace_{i\in I}$.
The injectivity of the   mapping $\alpha_{\Phi}$   leads to the reconstruction of every signal  in $\mathcal{H}$ up to a constant phase factor from the modulus   of its frame coefficients.

We denote by  $\tau_{w}$  the coarsest topology
 on $\hat{\mathcal{H}}$ that makes the function family $\{\rho_{y}\}_{y\in  \mathcal{H}}$, defined as $$\rho_{y}(\hat{x})=\vert\langle x,y\rangle\vert ,\quad (\hat{x}\in \hat{\mathcal{H}})$$  continuous. It is worth to mention that $\tau_{w}$ is Hausdorff. Indeed, $\hat{x}=\hat{y}$ if and only if $\vert\langle x,z\rangle\vert=\vert\langle y,z\rangle\vert$, for any $z\in \mathcal{H}$. Moreover we use the notation $\tau_{\pi}$ for the norm-quotient topology on $\hat{\mathcal{H}}$, induced by the quotient map $\pi: \mathcal{H}  \rightarrow \hat{\mathcal{H}}$; $x \mapsto\hat{x}$.

 There are    substantial topological differences between the quotient space $\hat{\mathcal{H}}$ and  the  projective Hilbert space, studied in the context of quantum mechanics, in topological point of view. Indeed, by considering $\mathcal{H}$ as a non-trivial Hilbert space, instead of  $\mathcal{H}^{*}$, one can consider only the unit sphere of  $\mathcal{H}$ and this is a milestone in the study of
the topological structures on  $P(\mathcal{H})$. In fact, with this attitude  $P(\mathcal{H})$ can be identified with the set of the pure quantum states which are the one dimensional orthogonal projections as in $\{P_{y};  \quad\Vert y\Vert =1\}$.
  In particular, it is proved that the non-trivial  projective Hilbert space carries a natural topology and a natural measurable structure, as well \cite{qua1,  qua5, stup.}. In this paper, we consider a separable Hilbert space, and corresponding  to each bounded phase retrieval family of vectors $\Phi=\{\phi_{i}\}_{i\in I}$  in $\mathcal{H}$, define a new metric $d_{\Phi}$ and an initial topology $\tau_{\Phi}$ on the quotient space $\hat{\mathcal{H}}$.   We then study   its  topological   properties and  present a comprehensive   comparison of  the initial topology with the weak topology, the direct-limit topology, the topology induced by the metric  $d_{\Phi}$ and the Bures-Wasserstein distance  on $\hat{\mathcal{H}}$.
Our analysis shows that   the initial topology is strictly weaker than $\tau_{w}$, $\tau_{\pi}$  and the topology induced by  the metrics in the infinite dimensional case. However,  in the finite dimensional case, all these topologies coincide. Our findings provide a deeper understanding of the quotient space and its topological properties.

The contents of the sections of this  paper are as follows. In Section 2 we briefly   present some elementary concepts of frame theory and the phase retrieval problem.
  In Section 3,  we define the metric $d_{\Phi}$ on the quotient space $\hat{\mathcal{H}}$  associated with each phase retrieval frame $\Phi$ of $\mathcal{H}$. We  show that $d_{\Phi}$ is
 weaker than the Bures-Wasserstein distance and they   coincide  in the finite dimensional case.
Section 4 is devoted to a topological perspective on $\hat{\mathcal{H}}$. We introduce an initial  topology $\tau_{\Phi}$ on   $\hat{\mathcal{H}}$ associated to a phase retrieval family $\Phi$  of vectors in $\mathcal{H}$ and  we show  that  such a topology is strictly weaker than $\tau_{w}$ as well as the direct-limit topology.
  In addition, we show that $\tau_{\Phi}$ coincides with $\tau_{w}$  as well as $\tau_{d_{\Phi}}$, whenever $\Phi$ is a  phase retrieval frame  in a finite dimensional Hilbert space.
Finally, we construct a topometric space by using   the triple $(\hat{\mathcal{H}}, \tau_{\Phi},d_{\Phi})$. This    provides a new   point of view  on the properties of the projective Hilbert space.

\section{Preliminaries}
In what follows, we review some basic definitions and results related to frame theory and phase retrieval vectors in   Hilbert spaces. Throughout the paper, we suppose that $\mathcal{H}$ denotes a separable Hilbert space,  and $\mathcal{H}_{n}$ denotes an $n$-dimensional  real or complex Hilbert space,  we use $\mathbb{R}^{n}$ and $\mathbb{C}^{n}$ whenever it is necessary  to distinguish between them.  We also consider the notations; $I_{n}=\{1, 2, ..., n\}$,  $I$ as a  countable index set  and    $\{e_{i}\}_{i\in I_{n}}$ as the standard orthonormal basis of $\mathcal{H}_{n}$. When using two topologies $\tau_{1}$ and $\tau_{2}$ we use ($\tau_{1}\prec\tau_{2}$) $\tau_{1}\preceq\tau_{2}$ to show $\tau_{1}$ is (strictly) weaker than $\tau_{2}$.

A sequence of vectors
$\Phi:=\{\phi_{i}\}_{i\in I}$ in   $\mathcal{H}$ is called a  \textit{frame}  if there exist the constants
 $0<A_{\Phi}\leq B_{\Phi}<\infty$ such that
\begin{eqnarray}\label{Def frame}
A_{\Phi}\|f\|^{2}\leq \sum_{i\in I}\vert \langle f,\phi_{i}\rangle\vert^{2}\leq
B_{\Phi}\|f\|^{2},\qquad (f\in \mathcal{H}).
\end{eqnarray}
The constants $A_{\Phi}$ and $B_{\Phi}$ are called \textit{frame bounds}. The sequence  $\{\phi_{i}\}_{i\in I}$ is said to be  a \textit{Bessel sequence} whenever the right hand side of  (\ref{Def frame})  holds.
 We refer the reader for more detailed  information on frame theory to \cite{Alexeev, AAS,  Bol98, Chr08,     H, Kovaevi}.

In 2006, the  phase recovery problem was formulated in terms of a mathematical language which is closely related to frame theory \cite{Casazza}. In fact, the injectivity of the mapping $\alpha_{\Phi}$ is equivalent to the following definition:

\begin{definition}\cite{Casazza}\label{1}
A family of vectors $\Phi=\{\phi_{i}\}_{i\in I}$ in $\mathcal{H}$ does phase  retrieval if whenever
$f,g\in \mathcal{H}$ satisfy
\begin{eqnarray}\label{121}
\vert \langle f,\phi_{i}\rangle\vert=\vert \langle g,\phi_{i}\rangle\vert, \qquad
 (i\in I)
\end{eqnarray}
then there exists a scalar $\lambda$ with $\vert \lambda\vert =1$ so that $f= \lambda g$.
\end{definition}
Also, in \cite{Casazza}, the authors investigated the injectivity of   $\alpha_{\Phi}$ in the finite dimensional real Hilbert spaces and obtained the least number of  measurements which  suffice for the  injectivity in $n$-dimensional   Hilbert spaces. We first recall the spark of a matrix  is the size of the smallest linearly dependent subset of the columns  \cite{Alexeev}. And the spark of a collection of vectors in a  finite dimensional Hilbert space is considered as the spark of its synthesis matrix, that is a matrix where the frame vectors form the columns. Moreover, a frame $\Phi=\{\phi_{i}\}_{i\in I_{m}}$  in  $\mathcal{H}_{n}$  is called to be full spark whenever $ spark(\Phi)=n+1$.
\begin{proposition}\cite{Casazza}
If  $\Phi=\{\phi_{i}\}_{i\in I_{m}}$  does phase retrieval in $\mathbb{R}^{n}$, then $m\geq 2n-1$. If $m\geq 2n-1$ and $\Phi$ is full spark then $\mathcal{\phi}$  does phase retrieval. Moreover, $\{\phi_{i}\}_{i\in I_{2n-1}}$  does phase retrieval if and only if $\Phi$ is full spark.\end{proposition}
Recall that two families of frames $\Phi$ and $\Psi$ of $\mathcal{H}$  are called to be equivalent if there exists an invertible operator $U\in B(\mathcal{H})$ so that $\Phi=U\Psi$. It is shown that equivalent frames have the same phase retrieval property.

A family of vectors $\Phi=\{\phi_{i}\}_{i\in I}$ in  $\mathcal{H}$ has the complement property if for every $\sigma \subset I$ either $\overline{span}\{\phi_{i}\}_{i\in \sigma}=\mathcal{H}$ or $\overline{span}\{\phi_{i}\}_{i\in \sigma^{c}}=\mathcal{H}$.
 A main characterization of frames which do phase retrieval was presented for the finite dimensional real case in \cite{Casazza} and then for the infinite dimensional case  in \cite{cahill} which states that,
a family $\Phi=\{\phi_{i}\}_{i\in I}$  in a real Hilbert space $\mathcal{H}$ does phase retrieval if and only if it has the complement property.

\section{The metric structure of the quotient space $\hat{\mathcal{H}}$}
In \cite{Balan-zhu}, the Bures-Wasserstein  metric  has been considered on $ \hat{\mathcal{H}}$ as follows;
\begin{eqnarray}\label{metric}
D(\hat{x},\hat{y})=\min_{\theta}\Vert x-e^{i\theta}y\Vert= \sqrt{\Vert x\Vert^{2}+\Vert y\Vert^{2}-2\vert \langle x, y\rangle\vert}.
\end{eqnarray}
Then, the  bi-Lipschitz bounds  of the mapping $\alpha_{\Phi}$ were investigated by applying this metric.
In the sequel,  we define a new  metric on $\hat{\mathcal{H}}$ that is weaker than $D$. Let $\Phi=\{\phi_{i}\}_{i\in I}$ be a  bounded sequence   in  Hilbert space   $\mathcal{H}$. Then we define  the function
 $d_{\Phi}:\hat{\mathcal{H}} \times \hat{\mathcal{H}} \rightarrow \mathbb{R}$   by
\begin{eqnarray}\label{metricnew}
d_{\Phi}(\hat{x},\hat{y})= \sup_{j\in I}\min_{\theta}\vert \langle x-e^{i\theta}y, \phi_{j}\rangle\vert.
\end{eqnarray}
The next result ensures that $d_{\Phi}$ is  a pseudo metric  on $\hat{\mathcal{H}}$.
\begin{theorem}\label{newmetric}
Let $\Phi=\{\phi_{i}\}_{i\in I}$ be a bounded sequence of vectors  in     $\mathcal{H}$.
Then $d_{\Phi}$ is a pseudo metric on $\hat{\mathcal{H}}$. Moreover, $d_{\Phi}$ is metric if and only if $\Phi$ yields  phase retrieval.
\end{theorem}
\begin{proof}
Clearly $d_{\Phi}$ is well-defined and symmetric, also if $\hat{x}=\hat{y}$ then there exists $\theta$ so that $x=e^{i\theta}y$ and so $d_{\Phi}(\hat{x},\hat{y})=0$.  Moreover, for every  $ j\in I$ and  $\hat{y},\hat{z}\in \hat{\mathcal{H}}$,  there exists $\alpha_{j}\in [0,2\pi)$ so that
\begin{eqnarray*}
\min_{\alpha}\vert \langle y-e^{i \alpha}z, \phi_{j}\rangle\vert =\vert \langle y-e^{i\alpha_{j}}z, \phi_{j}\rangle\vert, \quad (j\in I).
\end{eqnarray*}
Then,
\begin{eqnarray*}
d_{\Phi}(\hat{x},\hat{z})&=&\sup_{j\in I}\min_{\theta}\vert \langle x-e^{i\theta}z, \phi_{j}\rangle\vert\\
&=& \sup_{j\in I}\min_{\theta}\vert \langle xe^{i\alpha_{j}}-e^{i(\theta+\alpha_{j})}z, \phi_{j}\rangle\vert\\
&=&  \sup_{j\in I}\min_{\theta}\vert \langle xe^{i\alpha_{j}}-e^{i(\theta+\alpha_{j})}z+e^{i\theta}y-e^{i\theta}y, \phi_{j}\rangle\vert\\
&\leq&  \sup_{j\in I}\min_{\theta}\left(\vert \langle x-e^{i(\theta-\alpha_{j})}y, \phi_{j}\rangle\vert+\vert \langle y-e^{i\alpha_{j}}z, \phi_{j}\rangle\vert\right)\\
&\leq& \sup_{j\in I} \min_{\theta}\vert \langle x-e^{i\theta}y, \phi_{j}\rangle\vert+\sup_{j\in I} \min_{\alpha}\vert \langle y-e^{i \alpha}z, \phi_{j}\rangle\vert\\
&=& d_{\Phi}(\hat{x},\hat{y})+d_{\Phi}(\hat{y},\hat{z}),
\end{eqnarray*}
for every $\hat{x},\hat{y},\hat{z}\in \hat{\mathcal{H}}$.
Hence, the mapping $d_{\Phi}$ is a pseudo metric on $\hat{\mathcal{H}}$. For moreover part,   note that   $d(\hat{x},\hat{y})=0$ if and only if for every $j\in I$ there exists $\theta_{j}$ so that $\vert \langle x-e^{i\theta_{j}}y, \phi_{j}\rangle\vert=0$, or equivalently for every $j\in I$ there exists $\theta_{j}$ so that $ \langle x, \phi_{j}\rangle=e^{i\theta_{j}}\langle y, \phi_{j} \rangle$. Thus, $d_{\Phi}(\hat{x},\hat{y})=0$ if and only if $\vert \langle x, \phi_{j}\rangle\vert=\vert \langle y, \phi_{j}\rangle\vert$, for all $j\in I$. Hence, $d_{\Phi}$ is a metric on $\hat{\mathcal{H}}$  if and only if $\Phi=\{\phi_{i}\}_{i\in I}$ yields  phase retrieval.
\end{proof}
 \begin{proposition}\label{flem}
Suppose that $\Phi=\{\phi_{i}\}_{i\in I}$ is  a bounded  sequence  of  phase retrieval vectors  in     $\mathcal{H}$. Then
 $$d_{\Phi}(\hat{x},\hat{y})= \sup_{j\in I}\left\vert \vert \langle x,\phi_{j}\rangle\vert-\vert \langle y,\phi_{j}\rangle\vert\right\vert.$$
\end{proposition}
 \begin{proof}
We note that
\begin{eqnarray*}
d_{\Phi}^{2}(\hat{x},\hat{y})&=&\sup_{j\in I}\min_{\theta}\vert \langle x-e^{i\theta}y, \phi_{j}\rangle\vert^{2}\\
&=&\sup_{j\in I}\min_{\theta} \left( \vert \langle x, \phi_{j}\rangle\vert^{2}+\vert \langle  y, \phi_{j}\rangle\vert^{2}-2Re( e^{i\theta}\langle y, \phi_{j}\rangle\langle \phi_{j}, x\rangle)\right)\\
&=&\sup_{j\in I}\left[\vert \langle x, \phi_{j}\rangle\vert^{2}+\vert \langle  y, \phi_{j}\rangle\vert^{2}-2\max_{\theta} \left(Re( e^{i\theta}\langle y, \phi_{j}\rangle\langle \phi_{j}, x\rangle)\right)\right].
\end{eqnarray*}
Assuming $\langle y, \phi_{j}\rangle\langle \phi_{j}, x\rangle=\vert \langle y, \phi_{j}\rangle\langle \phi_{j}, x\rangle\vert e^{i\alpha_{j}}$ for some $\alpha_{j}\in [0,2\pi)$ we get
\begin{eqnarray*}
d_{\Phi}^{2}(\hat{x},\hat{y})
&=&\sup_{j\in I}\left[\vert \langle x, \phi_{j}\rangle\vert^{2}+\vert \langle  y, \phi_{j}\rangle\vert^{2}-2\max_{\theta} \left(Re( e^{i(\alpha_{j}+\theta)}\vert\langle y, \phi_{j}\rangle\langle \phi_{j}, x\rangle\vert)\right)\right]\\
&=&\sup_{j\in I}\left[\vert \langle x, \phi_{j}\rangle\vert^{2}+\vert \langle  y, \phi_{j}\rangle\vert^{2}-2\vert\langle y, \phi_{j}\rangle\vert \vert\langle \phi_{j}, x\rangle\vert \right]\\
&=&\sup_{j\in I}\left(  \vert \langle x,\phi_{j}\rangle\vert-\vert \langle y,\phi_{j}\rangle\vert \right)^{2}.
\end{eqnarray*}
\end{proof}

If $\Phi$ is a phase retrieval Bessel sequence with an upper bound $B$, then
\begin{eqnarray}\label{dweakerD}
d_{\Phi}(\hat{x},\hat{y})&=& \sup_{j\in I}\min_{\theta}\vert \langle x-e^{i\theta}y, \phi_{j}\rangle\vert\\
&\leq& \sqrt{B}\min_{\theta}\Vert x-e^{i\theta}y\Vert=\sqrt{B}D(\hat{x},\hat{y}).
\end{eqnarray}
That is,  $D$ dominates  $d_{\Phi}$.
Now,
 suppose that $\Phi=\{\phi_{i}\}_{i\in I}$  is a bounded sequence  in  Hilbert space $\mathcal{H}$.
Define
\begin{eqnarray*}
\mathfrak{D}(\hat{x},\hat{y})=\min_{\theta}\sup_{j\in I}\vert \langle x-e^{i\theta}y, \phi_{j}\rangle\vert.
\end{eqnarray*}
 The mapping $\mathfrak{D}$ is well-defined due to the mapping $f:[0,2\pi] \rightarrow \mathbb{R}$ defined by $f(\theta)=
 \sup_{j\in I}\vert \langle x-e^{i\theta}y, \phi_{j}\rangle\vert$ is lower semi continuous. In addition, lower semi continuous functions attains their  minimum on compact spaces  \cite{Laffer}, so there exists $\theta_{0}$ so that $\mathfrak{D}(\hat{x},\hat{y})=\sup_{j\in I}\vert \langle x-e^{i\theta_{0}}y, \phi_{j}\rangle\vert$.
  Moreover, $\mathfrak{D}$ defines a pseudo-metric on $\hat{\mathcal{H}}$. And, if $\Phi$ is a phase retrieval frame  an analogous  approach to the proof of Theorem \ref{newmetric} shows that $(\hat{\mathcal{H}}, \mathfrak{D})$ is a metric space.

 Assume that $\Phi=\{\phi_{i}\}_{i\in I}$  is a frame with frame bounds $A$ and $B$, respectively and $x,y\in \mathcal{H}_{n}$, then $\mathfrak{D}(\hat{x},\hat{y})\leq  \sqrt{B}D(\hat{x},\hat{y})$ . Furthermore, in finite dimensional case, $\mathfrak{D}$ is equivalent to $D$. Indeed,   considering $\Phi=\{\phi_{i}\}_{i\in I_{m}}$ as a frame in  $\mathcal{H}$ and $\theta_{0}=\arg  \min \max_{j\in I_{m}}\vert \langle x-e^{i\theta}y, \phi_{j}\rangle\vert$
we get
\begin{eqnarray*}
D(\hat{x},\hat{y}) &=&\min_{\theta}\Vert  x-e^{i\theta}y \Vert\\
&\leq&
\Vert  x-e^{i\theta_{0}}y \Vert\\
&\leq& \dfrac{1}{\sqrt{A}} \left(\sum_{j\in I_{m}}\vert \langle x-e^{i\theta_{0}}y, \phi_{j}\rangle\vert^{2}\right)^{1/2}\\
&\leq&\sqrt{\dfrac{m}{A}} \max_{j\in I_{m}}\vert \langle x-e^{i\theta_{0}}y, \phi_{j}\rangle\vert\\
&\leq&\sqrt{\dfrac{m}{A}}  \mathfrak{D}(\hat{x},\hat{y}).
\end{eqnarray*}
In the next example we compute $D$, $d_{\Phi}$ and $\mathfrak{D}$ for a phase retrieval frame in $\mathbb{C}^{2}$.
\begin{example}
Consider the following phase retrieval frame for $\mathbb{C}^{2}$
\begin{eqnarray*}
\Phi=\{(1,0), (0,1), (1,1), (1,i)\}.
\end{eqnarray*}
Fix $n ,m\in \mathbb{N}$ and take $x=(n,0)$, $y=(0,m)$. Then we obtain
\begin{eqnarray*}
D(\hat{x},\hat{y})= \min_{\theta}\Vert x-e^{i\theta}y\Vert=\min_{\theta}\Vert (n, -me^{i\theta})\Vert=\sqrt{n^{2}+m^{2}}.
\end{eqnarray*}
On the other hand,
\begin{eqnarray*}
d_{\Phi}(\hat{x},\hat{y})&=& \max_{1\leq j\leq 4}\min_{\theta}\vert \langle x-e^{i\theta}y, \phi_{j}\rangle\vert \\
&=& \max_{1\leq j\leq 4}\min_{\theta}\vert \langle (n, -me^{i\theta}), \phi_{j}\rangle\vert \\
&=&\max\{n, m, \vert m-n\vert, \vert m-n\vert\}\\
&=&\max\{n, m\}.
\end{eqnarray*}
Hence, we get
\begin{eqnarray*}
D^{2}(\hat{x},\hat{y})=m^{2}+n^{2}  \leq 2d_{\Phi}^{2}(\hat{x},\hat{y}).
\end{eqnarray*}
Furthermore,
\begin{eqnarray*}
\mathfrak{D}(\hat{x},\hat{y})=\min_{\theta}\max\left\{n,m,\sqrt{n^{2}+m^{2}-2nm cos\theta},\sqrt{n^{2}+m^{2}+2nmsin\theta}\right\}.
\end{eqnarray*}
So,
\begin{eqnarray*}
\mathfrak{D}(\hat{x},\hat{y}) =\min_{\theta}\begin{cases}
\begin{array}{ccc}
\sqrt{n^{2}+m^{2}+2nm sin\theta}& \;
{-sin^{-1}(n/2m)\leq \theta \leq \dfrac{3\pi}{4}}, \\
\sqrt{n^{2}+m^{2}-2nm cos\theta}& \; {\dfrac{3\pi}{4}}\leq \theta \leq \dfrac{3\pi}{2}, 0\leq cos\theta \leq \dfrac{n}{2m}, \\
\max\{n,m\}
& \; {\textit{otherwise}}.
\end{array}
\end{cases}
\end{eqnarray*}
 Thus,   $\mathfrak{D}(\hat{x},\hat{y})= \max\{n,m\} =d_{\Phi}(\hat{x},\hat{y})$.
\end{example}
As the final result  of this section we prove that in the finite dimensional case the quotient  space $\hat{\mathcal{H}}$ is  a  complete metric space.
  \begin{theorem}\label{closerange}
Let $\Phi=\{ \phi_{i}\}_{i\in I_{m}}$ does  phase retrieval  frame  in     $\mathcal{H}_{n}$. Then, $(\hat{\mathcal{H}}_{n},d_{\Phi})$ is a complete metric space.
\end{theorem}
\begin{proof}
It is sufficient to show that the mapping $\alpha_{\Phi}$ is closed range.
Assume that $\{\hat{x}_{n}\}_{n}$ is a Cauchy sequence in $\hat{\mathcal{H}}$, then by Proposition \ref{flem}  the sequence $\{\vert \langle x_{n}, \phi_{j}\rangle \vert\}_{n}$ is a Cauchy sequence in $\mathbb{R}$, for all $j\in I$. So
there exist  positive elements $a_{j}$, $j\in I$  so that $\lim_{n\rightarrow \infty}\vert \langle x_{n}, \phi_{j}\rangle \vert=a_{j}$.
Due to the sequence $\{\vert \langle x_{n}, \phi_{j}\rangle \vert\}_{j\in I}\in R(\alpha_{\Phi})$, for every $n$     there exists   $\hat{x}\in \hat{\mathcal{H}}$ so that $$\lim_{n\rightarrow \infty}\vert \langle x_{n}, \phi_{j}\rangle \vert=\vert \langle x, \phi_{j}\rangle \vert,$$  for all $j\in I$. That means, $\lim_{n\rightarrow \infty}d_{\Phi}(\hat{x}_{n},\hat{x} )=0$.

 Suppose now that $\{\hat{x}_{n}\}_{n}$ is a  sequence in $\hat{\mathcal{H}}$ so that $\lim_{n\rightarrow \infty}\{\vert \langle x_{n}, \phi_{j}\rangle \vert\}_{j\in I}=\{a_{j}\}_{j\in I}$, where
$\{a_{j}\}_{j\in I}$ is a sequence of positive numbers. Hence, we get  $$\lim_{n\rightarrow \infty}\vert \langle x_{n}, \phi_{j}\rangle \vert=a_{j} , \quad (j\in I).$$ On the other hand, there exist $\{\theta_{j,n}\}_{n}\in [0,2\pi)$ such  that  $$\langle x_{n},\phi_{j}\rangle=\vert\langle x_{n},\phi_{j}\rangle \vert e^{i\theta_{j,n}}, \quad (n\in \mathbb{N}).$$
The sequence $\{\theta_{j,n}\}_{n}$ has a convergent subsequence in $[0,2\pi)$ and so  we may assume that there exists $\theta_{j}$ so that $\lim_{k\rightarrow \infty}  \langle x_{n_{k}}, \phi_{j}\rangle =a_{j}e^{i\theta_{j}}$. Consider $x=\sum_{j\in I}a_{j}e^{i\theta_{j}}S_{\phi}^{-1}\phi_{j}$. Then we get
\begin{eqnarray*}
\lim_{n\rightarrow \infty}\vert \langle x_{n}, \phi_{i}\rangle \vert&=&\lim_{k\rightarrow \infty}\vert \langle x_{n_{k}}, \phi_{i}\rangle \vert\\&=&
\left\vert \lim_{k\rightarrow \infty}\left\langle \sum_{j\in I}\langle x_{n_{k}}, \phi_{j} \rangle S_{\phi}^{-1}\phi_{j}, \phi_{i}\right\rangle \right\vert\\
&=&\left\vert \left\langle \sum_{j\in I}\lim_{k\rightarrow \infty}  \langle  x_{n_{k}}, \phi_{j}\rangle S_{\phi}^{-1}\phi_{j}, \phi_{i}\right\rangle \right\vert\\
&=&\left\vert\left \langle \sum_{j\in I} a_{j}  e^{i\theta_{j}} S_{\phi}^{-1}\phi_{j}, \phi_{i}\right\rangle \right\vert\\
&=&\vert \langle x, \phi_{i}\rangle \vert.
\end{eqnarray*}
This follows the desired result.
\end{proof}
 The quotient space $ \hat{\mathcal{H}}$  equipped with the metric $d_\Phi$ is generally incomplete in infinite dimensions. We demonstrate this by constructing a Cauchy sequence that does not converge to any element in $ \hat{\mathcal{H}}$.
\begin{example}
Assume that  $\mathcal{H} = \ell^2(\mathbb{N})$ and $\Phi = \{e_j\}_{j=1}^\infty$ is  the standard orthonormal basis. Consider the sequence $x_n = \sum_{k=1}^n \frac{1}{\sqrt{k}} e_k$. Then $\{\hat{x}_n\}$ is a Cauchy sequence in $(\hat{\mathcal{H}}, d_\Phi)$ but does not converge to any $\hat{x} \in \hat{\mathcal{H}}$. More precisely, $$d_\Phi(\hat{x}_n, \hat{x}_m) = \sup_{j \in \mathbb{N}} \left| |\langle x_n, e_j \rangle| - |\langle x_m, e_j \rangle| \right|\sup_{n < j \le m} \frac{1}{\sqrt{j}} = \frac{1}{\sqrt{n+1}},$$
since $\langle x_n, e_j \rangle = 1/\sqrt{j}$ for $j \le n$ and $0$ otherwise.
As $n \to \infty$, this distance approaches $0$, so $\{\hat{x}_n\}$ is Cauchy.
If $\hat{x}_n \to \hat{x}$ for some $x \in \ell^2(\mathbb{N})$, then for each fixed $j$:
$$ |\langle x, e_j \rangle| = \lim_{n \to \infty} |\langle x_n, e_j \rangle| = \frac{1}{\sqrt{j}}.$$
However, $\sum_{j=1}^\infty |\langle x, e_j \rangle|^2 = \sum_{j=1}^\infty \frac{1}{j}$ diverges. Thus, no such $x \in \ell^2(\mathbb{N})$ exists, proving the space is incomplete.
\end{example}
\section{Topological  point of view}
In  this section, we define an initial topology $\tau_{\Phi}$ on the quotient space $\hat{\mathcal{H}}$ associated to a phase retrieval family of vectors in $\mathcal{H}$. The main goal  is to investigate the topological structure of $\hat{\mathcal{H}}$ and to  make a  comparison between  $\tau_{\Phi}$ and some known topologies on $\hat{\mathcal{H}}$.   In particular, we prove that $\hat{\mathcal{H}}$ is paracompact, and
$\tau_{\Phi}$ coincides with $\tau_{w}$   in the finite dimensional case if and only if $\Phi$ does phase retrieval.

Let    $\Phi=\{\phi_{i}\}_{i\in I}$ be a sequence  in   Hilbert space $\mathcal{H}$. Then the family $\{\rho_{\phi_{i}}\}_{i\in I}$, where $\rho_{\phi_{i}}: \hat{\mathcal{H}} \rightarrow \mathbb{R}$, defined by $\rho_{\phi_{i}}\hat{x}=\vert\langle x,\phi_{i}\rangle\vert$, induces an initial topology $\tau_{\Phi}$  on    $ \hat{\mathcal{H}}$. This topology is  the coarsest topology on
$ \hat{\mathcal{H}}$ that makes the function family $\{\rho_{\phi_{i}}\}_{i\in I}$ continuous with the subbases $\{\rho^{-1}_{\phi_{i}}(r_{1}, r_{2}) : r_{1}, r_{2}\in \mathbb{R}^{+},i\in I\}$. Studying the structure of $\tau_{\Phi}$ is the main aim of this section. We first note that in the case $\Phi=\{\phi_{i}\}_{i\in I}$ does  phase retrieval   in     $\mathcal{H}$, the space  $(\hat{\mathcal{H}}, \tau_{\Phi})$ is $T_{1}$. Indeed,  it is sufficient to show that $\{\hat{x}\}^{c}$ is open for every $\hat{x}\in \hat{\mathcal{H}}$. Consider $\hat{y}\in \{\hat{x}\}^{c}$ then $\hat{x}\neq \hat{y}$ and so  there exists $i\in I$ such that $\vert \langle x,\phi_{i}\rangle\vert\neq\vert \langle y,\phi_{i}\rangle\vert$,  we may assume that
\begin{eqnarray*}
\vert \langle y,\phi_{i}\rangle\vert < \vert \langle x,\phi_{i}\rangle\vert.
\end{eqnarray*}
Hence, $\hat{y}\in \rho^{-1}_{\phi_{i}}(0, \vert \langle x,\phi_{i}\rangle\vert)\subseteq \{\hat{x}\}^{c}$, i.e., $\{\hat{x}\}^{c}$ is an open set in $\hat{\mathcal{H}}$.

 As   mentioned above, a fundamental  classification of phase retrieval frames  in real Hilbert spaces in based on  the complement property. The following lemma  presents a new characterization of the phase retrieval property in both real and complex cases. Indeed,    we show that the phase retrieval property of $\Phi$ is equivalent to $(\hat{\mathcal{H}}, \tau_{\Phi})$ being a Hausdorff space.  While the argument is simple, it provides fresh insight into the problem.
\begin{proposition}\label{Haus}
$\Phi=\{\phi_{i}\}_{i\in I}$ does  phase retrieval   in     $\mathcal{H}$ if and only if  $(\hat{\mathcal{H}}, \tau_{\Phi})$ is Hausdorff.
\end{proposition}
\begin{proof}
Notice that   $(\hat{\mathcal{H}}, \tau_{\Phi})$ is Hausdorff  if and only if $\{\rho_{\phi_{i}}\}_{i\in I}$ separates the points of $\hat{\mathcal{H}}$ that is   equivalent to $\Phi$ yields  phase retrieval.
\end{proof}

\begin{lemma}\label{converges}
Let $\Phi=\{\phi_{i}\}_{i\in I}$ yields  phase retrieval   in     $\mathcal{H}$, $x\in \mathcal{H}$ and $\{\hat{x}_{n}\}_{n=1}^{\infty}\subseteq \hat{\mathcal{H}}$. Then  the followings hold;
\begin{itemize}
\item[(i)]  $\lim_{n} \hat{x}_{n}=\hat{x}$ in $\tau_{w}$   if and only if  $\lim_{n}\vert \langle x_{n},y\rangle\vert =\vert \langle  x,y \rangle\vert$, for all $y\in \mathcal{H}$.
\item[(ii)] $ \lim_{n} \hat{x}_{n}=\hat{x}$ in $\tau_{\Phi}$   if and only if  $\lim_{n}\vert \langle x_{n},\phi_{i}\rangle\vert =\vert \langle  x,\phi_{i}\rangle\vert$, for all $i\in I$.
\end{itemize}
\end{lemma}
\begin{proof}
$(i)$ follows from a  known result in the context of topology induced by a family of functions, see for instance \cite{book1}.
To obtain $(ii)$,
it is sufficient to note that $\tau_{\Phi}$ is indeed the  topology induced by the family of functions $\{\rho_{\phi_{i}}\}_{i\in I}$, where $\rho_{\phi_{i}}: \hat{\mathcal{H}} \rightarrow \mathbb{R}$, defined by $\rho_{\phi_{i}}\hat{x}=\vert\langle x,\phi_{i}\rangle\vert$.
\end{proof}
\begin{lemma}\label{finer}
Let   $\mathcal{H}$ be an infinite dimensional Hilbert space and $\Phi=\{\phi_{i}\}_{i\in I}$ be a bounded sequence  of vectors   in     $\mathcal{H}$, which  does  phase retrieval.
Then the topology induced by the metric $d_{\Phi}$   on $\hat{\mathcal{H}}$ is finer than $\tau_{\Phi}$.
\end{lemma}
\begin{proof}
Suppose that
 $\{\hat{x}_{n}\}_{n=1}^{\infty}$ is  a sequence which
  converges to $\hat{x}$  in $(\hat{\mathcal{H}}, d_{\Phi})$. Then $$\lim_{n}\sup_{j\in I}\left\vert \vert \langle x_{n},\phi_{j}\rangle\vert-\vert \langle x,\phi_{j}\rangle\vert\right\vert=0,$$  by Proposition  \ref{flem}. Hence,
 $\lim_{n}\left\vert \vert \langle x_{n},\phi_{j}\rangle\vert-\vert \langle x,\phi_{j}\rangle\vert\right\vert=0$, for all $j\in I$. Thus by Lemma \ref{converges} we get the desired result.
\end{proof}
Applying Lemma \ref{finer} along with $(3.4)$ follows that
\begin{eqnarray}\label{comp}
\tau_{\Phi} \preceq  \tau_{d_{\Phi}} \preceq  \tau_{D}\cong
\tau_{\pi}.
\end{eqnarray}
 Also, by using Proposition  \ref{flem} and Lemma \ref{converges} we easily obtain  the following result;
\begin{corollary}
Let $\Phi=\{\phi_{i}\}_{i\in I_{m}}$ be a phase retrieval frame  in     $\mathcal{H}_{n}$.
Then the topology induced by metric $d_{\Phi}$ and   $\tau_{\Phi}$ are equivalent    on $\hat{\mathcal{H}}_{n}$.
\end{corollary}
\begin{proof}
Suppose  the sequence $\{\hat{x}_{k}\}_{k=1}^{\infty}$
  converges to $\hat{x}$  in $(\hat{\mathcal{H}_{n}}, d_{\Phi})$, then using Proposition  \ref{flem} and  Lemma \ref{converges} (ii) we deduce that $\{\hat{x}_{k}\}_{n=1}^{\infty}$
  converges to $\hat{x}$  in $(\hat{\mathcal{H}_{n}}, \tau_{\Phi})$. This along with    $(\ref{comp})$ follows the result.
\end{proof}
In the sequel, we would like to  compare the topologies $\tau_{w}$ and $\tau_{\Phi}$  on $ \hat{\mathcal{H}}$. We first
  presents a description of  $\tau_{w}$   on $ \hat{\mathcal{H}}$. The proof is straightforward and is  left to the reader.
\begin{proposition}\label{Haus10}
Let $\mathcal{H}$ be a separable Hilbert space. The topology $\tau_{w}$  on  $ \hat{\mathcal{H}}$ coincides with  the weak quotient topology induced by the quotient map $\pi:(\mathcal{H},w)\rightarrow \hat{\mathcal{H}}$, where $w$   denotes the weak topology on Hilbert space $\mathcal{H}$.
\end{proposition}
We  note that the inclusion $\tau_{\Phi}\preceq \tau_{w}$  holds
trivially,  while   Proposition \ref{Haus} implies  that even the
condition $\overline{\text{span}}(\Phi) = \mathcal{H}$ is  not
sufficient to ensure that $\tau_\phi$ and $\tau_w$ are equivalent.
For instance, consider   $\mathcal{H} = \mathbb{C}^2$ and $\Phi =
\{e_1, e_2\}$, the standard orthonormal basis. Then, $\tau_\phi$
is not Haudorff. Specifically, consider $x = (1, 1)$ and $y= (1,
-1)$, we have  $|\langle x, e_j \rangle| = |\langle y, e_j
\rangle|$ for  $j=1,2$, while $\hat{x}\neq \hat{y}$. In short:
$$\tau_{\Phi}\preceq\tau_{w}\preceq  \tau_{\pi}.$$
We  point out that $(\hat{\mathcal{H}}, \tau_{\Phi})$  cannot be compact or discrete. To be more precise,
  the open cover $$ \{\hat{x}: \vert \langle x, \phi_{i}\rangle \vert < k\}_{i\in I ,k\in \mathbb{N}}$$ of $  \hat{\mathcal{H}}$ contains no finite subcover. In addition, $(\hat{\mathcal{H}}, \tau_{\pi})$ (and so $(\hat{\mathcal{H}}, \tau_{\Phi})$) cannot be discrete. Indeed, if $(\hat{\mathcal{H}}, \tau_{\pi})$  is discrete, then $\{\hat{x}\}$ is open in $\tau_{\pi}$, for each $\hat{x}\in \hat{\mathcal{H}}$. So, $$\pi^{-1}\{\hat{x}\}=\{xe^{i\theta}: \quad \theta\in [0,2\pi)\}$$ is open in $\mathcal{H}$. Hence, considering the continuous mapping $$f: \mathcal{H}\times [0,2\pi) \rightarrow \mathcal{H}; \quad (t,\theta)\mapsto e^{i\theta}t$$ follows that $\{x\}\times [0,2\pi)$ is an open subspace in $\mathcal{H}\times [0,2\pi)$. And consequently, $\{x\}$ is open in $\mathcal{H}$, for each $x\in \mathcal{H}$, i.e., $\mathcal{H}$ is discrete that is a contradiction.

In the following example, we observe that  $\tau_{\Phi}$ can be strictly weaker than the metric $d_{\Phi}$ and  the topology $\tau_{w}$.
\begin{example}\label{strict}
Let  $\Phi=\{e_{i}+e_{j}\}_{i<j}$, where $\{e_{i}\}_{i=1}^{\infty}$ is the canonical orthonormal basis of  $l^{2}$. Then $\Phi$ does phase retrieval in $l^{2}$, see \cite{Casazza22}. Now consider the sequence $ \{x_{n}\}_{n=1}^{\infty}=\{ne_{n}\}_{n=1}^{\infty}$. One can see by simply computing  that $\lim_{n}\vert \langle x_{n},\phi_{i}\rangle\vert =0$, for all $i$. However, taking $y=\{1/n\}_{n=1}^{\infty}$ implies that $\vert \langle x_{n},y\rangle\vert =1$, for all $n$. That means $ \{x_{n}\}_{n=1}^{\infty}$ does not converges  in $\tau_{w}$. Furthermore, $\lim_{n}d_{\Phi}(\hat{x}_{n},0)= \lim_{n}\sup_{i<j}\vert \langle  ne_{n},e_{i}+e_{j}\rangle\vert =\infty.$
\end{example}
In the next result, we apply the Urysohn's metrization theorem to show that $\hat{\mathcal{H}}$ along with the initial topology $\tau_{\Phi}$ is metrizable.
\begin{corollary}
Suppose $\Phi$ does phase retrieval in $\mathcal{H}$. Then $(\hat{\mathcal{H}}, \tau_{\Phi})$ is metrizable.
\end{corollary}
\begin{proof}
We note that $(\hat{\mathcal{H}}, \tau_{\Phi})$ is  a second countable space as well as a Hausdorff space by using Proposition \ref{Haus}. On the other hand, since $\Phi$ does phase retrieval,  the collection $\{\rho_{\phi_{j}}\}_{j\in I}$ is a separating  family of functions from $\hat{\mathcal{H}}$ into $\mathbb{R}$, and so $(\hat{\mathcal{H}}, \tau_{\Phi})$ is regular,
by  Proposition 2.4.8 of \cite{book1}.
Hence, the Urysohn's metrization theorem assures that  $(\hat{\mathcal{H}}, \tau_{\Phi})$ is a metrizable topological space.
\end{proof}
\subsection{Finite dimensional case}
 A straightforward argument shows that if $\Phi=\{\phi_{i}\}_{i=1}^{3}$ is a phase retrieval frame in $\mathbb{R}^{2}$, then   $\tau_{\Phi}$ and $\tau_{w}$ are equivalent. More precisely, we may consider without loss of generality $\phi=\{e_{1}, e_{2},  \alpha e_{1}+ \beta e_{2}\}$ for non-zero values $\alpha, \beta$ \cite{Arefi5}. Suppose the sequence $\{\hat{x}_{n}\}_{n=1}^{\infty}=\{[(a_{n}, b_{n})]\}_{n=1}^{\infty}$  converges to  $\hat{x}=[(a,b)]$  in $(\hat{\mathbb{R}^{2}}, \tau_{\Phi})$. Then by using Lemma \ref{converges} we get
\begin{eqnarray*}
  \lim_{n} a_{n}^{2}= a^{2}, \quad \lim_{n} b_{n}^{2}= b^{2}, \quad   \lim_{n} (\alpha a_{n} +\beta b_{n})^{2}=(\alpha a +\beta b)^{2}.
\end{eqnarray*}
The above equations easily follow that $\lim_{n} a_{n}b_{n}= ab$. Therefore, for every $y=(y_{1},y_{2})\in \mathbb{R}^{2}$
$$  \lim_{n} \left(a_{n}y_{1}+b_{n}y_{2}\ \right)^{2}= \left(a y_{1}+b y_{2} \right)^{2}.$$   Thus, $\{\hat{x}_{n} \}_{n=1}^{\infty}$ converges to $\hat{x}$ in $(\hat{\mathbb{R}^{2}},\tau_{w})$,  by Lemma \ref{converges}.
 \begin{example}
Let $\Phi=\{e_{1}, e_{2}, e_{3}, \phi_{4}, \phi_{5}\}$ does  phase retrieval   in     $\mathbb{R}^{3}$, where $\phi_{4}=(t_{1},t_{2},t_{3})$ and $\phi_{5}=(s_{1},s_{2},s_{3})$.
Then $\tau_{\Phi}$ and $\tau_{w}$ are equivalent.
To see this,
suppose that  $\{\hat{x}_{n}\}_{n=1}^{\infty}=\{[(a_{n}, b_{n}, c_{n})]\}_{n=1}^{\infty}$ is a sequence in $\hat{\mathbb{R}^{3}}$. If $\lim_{n}\hat{x}_{n}=\hat{x}$ in $\tau_{\Phi}$ where $\hat{x}=[(a,b,c)]$, then by using Lemma \ref{converges}    we get
\begin{eqnarray}\label{equa}
&& \lim_{n} a_{n}^{2}= a^{2}, \quad \lim_{n} b_{n}^{2}= b^{2},\quad \lim_{n} c_{n}^{2}= c^{2}
\end{eqnarray}
\begin{eqnarray*}
&&  \lim_{n} \left(a_{n}b_{n}t_{1}t_{2}+a_{n}c_{n}t_{1}t_{3}+b_{n}c_{n}t_{2}t_{3}\right)=abt_{1}t_{2}+ac t_{1}t_{3}+bc t_{2}t_{3}\\
&&  \lim_{n} \left(a_{n}b_{n}s_{1}s_{2}+a_{n}c_{n}s_{1}s_{3}+b_{n}c_{n}s_{2}s_{3}\right)=abs_{1}s_{2}+ac s_{1}s_{3}+bc s_{2}s_{3}.
\end{eqnarray*}
This easily follows that
\begin{eqnarray}\label{equa22}
&&  \lim_{n} \left((a_{n}b_{n}-ab)t_{2}s_{2}(t_{1}s_{3}-s_{1}t_{3})+(a_{n}c_{n}-ac)t_{3}s_{3}(t_{1}s_{2}-s_{1}t_{3})\right)=0\\
&&  \lim_{n} \left((a_{n}b_{n}-ab)t_{1}s_{1}(t_{3}s_{2}-s_{3}t_{2})+(b_{n}c_{n}-bc)t_{3}s_{3}(t_{1}s_{2}-s_{1}t_{3})\right)=0.
\end{eqnarray}
Notice that the phase retrieval property of $\Phi$ guarantees that all coefficients in $(4.2)$ and $(4.3)$  are non-zero \cite{Arefi5}. Moreover, it   follows from $(\ref{equa})$ that $\lim_{n} a_{n}^{2}b_{n}^{2}= a^{2}b^{2}$. In particular, every convergent subsequence of $\{a_{n}b_{n}\}_{n=1}^{\infty}$ converges to $ab$  or $-ab$. We claim that
\begin{equation}\label{liminf}
\limsup_{n}a_{n}b_{n}=ab=\liminf_{n}a_{n}b_{n}.
\end{equation}
 Otherwise, $a,b,c$ are non-zero and  there exists a subsequence $\{a_{n_k}b_{n_k}\}_{k=1}^{\infty}$ so that $\lim_{n}a_{n_k}b_{n_k}=-ab$. Applying $(4.2)$ and $(4.3)$ and keeping in mind that their coefficients are non-zero  imply that $$ \lim_{n}a_{n_k}c_{n_k}=-ac, \quad  \lim_{n}b_{n_k}c_{n_k}=-bc.$$ Therefore $ \lim_{n}a^{2}_{n_k}b^{2}_{n_k}c^{2}_{n_k}=-a^{2}b^{2}c^{2}$, which is impossible. Thus, $(4.2)$, $(4.3)$ and $(\ref{liminf})$ show  that $\lim_{n}a_{n}c_{n}=ac$ and $\lim_{n}b_{n}c_{n}=bc$.
Combining these facts and $(\ref{equa})$, for every $y=(y_{1},y_{2},y_{3})$ we deduce that
$$  \lim_{n} \left(a_{n}y_{1}+b_{n}y_{2}+c_{n}y_{3} \right)^{2}= \left(ay_{1}+b y_{2}+c y_{3} \right)^{2}.$$  This means that $\{\hat{x}_{n} \}_{n=1}^{\infty}$ converges to $\hat{x}$ in $(\hat{\mathbb{R}^{3}},\tau_{w})$ by Lemma \ref{converges}, as required.
\end{example}
The above examples provide motivation for  investigating the equivalence of topologies $\tau_{\Phi}$ and $\tau_{w}$. Applying Lemma  \ref{converges}, we show that for a finite dimensional Hilbert space  $\mathcal{H}_{n}$, the topologies $\tau_{\Phi}$ and $\tau_{w}$  are equivalent  on $\hat{\mathcal{H}}_{n}$.
\begin{theorem}\label{mainfinite}
Let $\Phi=\{ \phi_{i}\}_{i\in I_{m}}$ be a frame for   the  finite dimensional Hilbert space $\mathcal{H}_{n}$. Then $\tau_{\Phi}$ and $\tau_{w}$  are equivalent if and only if
$\Phi$ does  phase retrieval   in     $\mathcal{H}_{n}$.
\end{theorem}
\begin{proof}
We note that $\tau_{w}$ is Hausdorff, so
in case $\tau_{\Phi}$ and $\tau_{w}$  are equivalent   in     $\mathcal{H}_{n}$, applying Proposition  \ref{Haus} one easily deduces that $\Phi=\{ \phi_{i}\}_{i\in I_{m}}$ does  phase retrieval   in     $\mathcal{H}_{n}$.

For the converse, first assume that $\Psi=\{e_{1}, ... , e_{n}, \psi_{n+1},...,  \psi_{m}\}$ does phase retrieval,  where $\{e_{i}\}_{i\in I_{n}}$  is the canonical orthonormal basis of  $\mathcal{H}_{n}$ and claim that  $\tau_{w}$ and $ \tau_{\Psi}$  coincide. To this end, it is sufficient to show that $\tau_{w}\preceq \tau_{\Psi}$, so  let   $\{\hat{x}_{k}\}_{k=1}^{\infty}$
  converges to $\hat{y}$  in $\tau_{\Psi}$, where $x_{k}=(x^{k}_{1}, ..., x^{k}_{n})$ and $y=(y_{1},...y_{n})$. Due to Lemma \ref{converges} we obtain $$\lim_{k}\vert x^{k}_{i}\vert=\vert y_{i}\vert , \quad (i\in I_{n}).$$ This easily follows that $\{ x_{k}\}_{k=1}^{\infty}$ is norm bounded. Put
\begin{equation}\label{Bnhat}
\hat{B}_{m}=\{\hat{x}\in \hat{\mathcal{H}}_{n}: \Vert x\Vert \leq m\}, \quad (m\in \mathbb{N}).
\end{equation}
Then, there exists $N\in \mathbb{N}$ such that $\hat{x}_{k}, \hat{x}\in \hat{B}_{N}$, for all $k\in \mathbb{N}$. We show that $\tau_{\Psi}$ and $\tau_{w}$ coincide on $\hat{B}_{m}$, for each $m$.
We first need to prove that via the relative topology $\hat{B}_{m}$  is compact, for each $m$.
To this end, we note that   the quotient  mapping $$\Lambda: (\mathcal{H}_{n}, \Vert .\Vert ) \rightarrow (\hat{\mathcal{H}}_{n}, \tau_{w}); \quad x \mapsto \hat{x}$$ is continuous, in which we consider the usual norm  topology on $\mathcal{H}_{n}$. More precisely, $\mathcal{H}_{n}$ is separable and  $\Lambda$ is sequentially continuous by Lemma \ref{converges}.
 Furthermore, for each $m$, $B_{m}=\{x: \Vert x\Vert \leq m\}$ is a compact subset of $(\mathcal{H}_{n},w)$. Hence, $$\hat{B}_{m}:=\Lambda(B_{m})=\{\hat{x}\in \hat{\mathcal{H}}_{n}: \Vert x\Vert \leq m\}$$ is compact in $(\hat{\mathcal{H}}_{n}, \tau_{w})$.
What is more,
the identity mapping $$I:(\hat{B}_{m} ,\tau_{w})\rightarrow (\hat{B}_{m} , \tau_{\Psi})$$ is continuous, due to $\tau_{\Psi}$ is weaker than $\tau_{w}$, by Lemma \ref{converges} (iii).
Thus, the set  $\hat{B}_{m}$  in $\tau_{w}$ is compact, for any $n\in \mathbb{N}$, moreover   $(\hat{\mathcal{H}}_{n}, \tau_{\Psi})$ is Hausdorff by Proposition \ref{Haus}.  Hence, $I$  has a continuous inverse. This  immediately follows that  $\tau_{w}$ and $\tau_{\Psi}$ coincide on $\hat{B}_{m}$. Therefore, the sequence  $\{\hat{x}_{k}\}_{k=1}^{\infty}$
  converges to $\hat{y}$  in $\tau_{w}$, and so $\tau_{w}\preceq \tau_{\Psi}$.

 Now let $\Phi=\{ \phi_{i}\}_{i\in I_{m}}$ be a phase retrieval frame in $\mathcal{H}_{n}$, then $\Phi$ contains a Riesz basis and therefore
 there exists a bounded invertible operator $U\in B(\mathcal{H}_{n})$ and the vectors $\{\psi_{i}\}_{i=n+1}^{m}$ such that we can write;
$$\Phi=\{Ue_{1}, ... ,U e_{n}, U\psi_{n+1},...,  U\psi_{m}\}.$$ To complete  the proof, we will show that $\tau_{\Phi}$ and $\tau_{w}$ coincide.
Since the phase retrieval property is preserved under invertible operators it follows that $\Psi=\{e_{i}\}_{i\in I_{n}}\cup \{\psi_{i}\}_{i=n+1}^{m}$ is a phase retrieval frame  and hence $\tau_{\Psi}$ and $\tau_{w}$ coincide by the above argument. Now if
 $\{\hat{x}_{n}\}_{n=1}^{\infty}\subseteq \hat{\mathcal{H}}_{n}$
  converges to $\hat{x}$  in $\tau_{\Phi}$, then by Lemma \ref{converges} we get
$$\lim _{n}\vert \langle  U^{*}x_{n} , \psi_{i} \rangle\vert =\lim _{n}\vert \langle x_{n} , \phi_{i}\rangle\vert=\vert \langle x , \phi_{i}\rangle\vert=\vert \langle U^{*}x , \psi_{i}\rangle\vert, \quad i\in I_{m}.$$ This assures that $\{\widehat{U^{*}x}_{n}\}_{n=1}^{\infty}$  in $\tau_{\Psi}$ and so in $\tau_{w}$ converges to $\widehat{U^{*}x}$. Thus, $$\lim _{n}\vert \langle x_{n} , z\rangle\vert=
\lim_{n} \vert \langle  U^{*}x_{n} , U^{-1}z\rangle\vert =\vert \langle U^{*}x , U^{-1}z\rangle\vert=\vert \langle x , z\rangle\vert,$$ for every $z\in \mathcal{H}_{n}$.
\end{proof}

\begin{remark}
\begin{itemize}
\item[(i)]
Suppose $\Phi$ is a phase retrieval frame  in a finite dimensional Hilbert space, then  by applying the previous results we obtain  $$\tau_{d_{\Phi}} \cong \tau_{D}\cong \tau_{\mathfrak{D}}\cong \tau_{\Phi}\cong \tau_{w}.$$   However, in the  infinite dimensional case the topology $\tau_{\Phi}$ and the topology induced by the metric $d_{\Phi}$ do not coincide, in general. Again, consider the phase retrieval frame $\phi$ as in
Example \ref{strict}. Then $$d_{\Phi}(\hat{e}_{n},0)= \sup_{i<j}\vert \langle  e_{n},e_{i}+e_{j}\rangle\vert =1.$$
A simple observation shows that $\lim_{n}\vert \langle  e_{n},e_{i}+e_{j}\rangle\vert =0$. Thus $d_{\Phi}$ and   $\tau_{\Phi}$ do not  coincide even  on $\hat{B}_{m}$, $m\in \mathbb{N}$.
\item[(ii)]Notice that the condition that $\Phi$ does phase retrieval in Theorem \ref{mainfinite} is necessary. For example consider the orthonormal basis $\Phi=\{(1,0), (0,1)\}$ in $\mathbb{R}^{2}$. Take the sequence $x_{k}=((-1)^{k},(-1)^{k+1})$. Then  $\{\hat{x}_{n}\}_{n=1}^{\infty}$
  converges to $\widehat{(1,1)}$  in $\tau_{\Phi}$, however does not converge in $\tau_{w}$. So, in this case $\tau_{\Phi}\prec\tau_{w}$.
\end{itemize}
\end{remark}
\subsection{Direct limit topology}

In a topological space $X$, a refinement of a cover $V$ of $X$ is a new cover $W$ of $X$ such that every set in $W$ is contained in some set in $V$. A paracompact space is a topological space  $X$ in which every open cover $V$  has an open refinement $W$ that is locally finite, i.e., each point in the space  has a neighborhood that intersects only finitely many of the sets in the collection $W$. For instance, every compact space is paracompact \cite{Munkres}.  In our setting, it is important to note that the quotient norm topology, which is equivalent to the topology $\tau_D$ on $\hat{\mathcal{H}}$, is metrizable and therefore naturally paracompact by Dieudonné’s Theorem \cite{dieudonne1944}.  However, as pointed out by Woronowicz \cite{woronowicz1970} and Komatsu \cite{Komatsu}, such a topology might not fully reflect the internal structure of the space.

In what follows, we consider the direct limit topology on $\hat{\mathcal{H}}$ and demonstrate that this topology not only preserves paracompactness but also ensures that $\hat{\mathcal{H}}$ is $\sigma$-compact. We also make a comparison between the topology $\tau_{w}$, the  initial topology $\tau_{\Phi}$  induced by a phase retrieval frame $\Phi$, and the  direct limit topology.
First, we recall that for a tower of topological spaces $X_{1}\subseteq X_{2}\subseteq \dots$, the direct limit topological space $\underrightarrow{\lim}X_{n}$ is the space $\cup_{n\in \mathbb{N}} X_{n}$ endowed with the topology $\tau_{\infty}$:
$$U \in \tau_{\infty} \iff U\cap X_{n} \text{ is open in } X_{n}, \forall n\in \mathbb{N}.$$
In the other words, the direct limit topology is the finest topology so that the inclusion mappings $X_{n}\rightarrow \underrightarrow{lim}X_{n}$, $n\in \mathbb{N}$, are continuous.
We also recall that, a tower  $X_{1}\subseteq X_{2}\subseteq ...$   of spaces is said to be closed if each $X_{n}$ is closed in $X_{n+1}$.
We refer the reader to \cite{book} for more details.

 In the following theorem, we formally prove that $\hat{\mathcal{H}}$ is a paracompact space, a crucial property that ensures the existence of partitions of unity within this inductive framework.
\begin{theorem} \label{paracompact}
The space   $\hat{\mathcal{H}}$ endowed by the  direct limit topology    is
 paracompact.
\end{theorem}
\begin{proof}
The main idea of the proof is to  show that  $\hat{\mathcal{H}}$ is  the direct limit of  a
 tower of
 compact spaces.
To this end, we take  $ \hat{B}_{n}=\{\hat{x}: \Vert x\Vert \leq n\}$, for $n\in \mathbb{N}.$ Then obviously
 $\hat{B}_{n}\subseteq \hat{B}_{n+1}$, for all $n\in \mathbb{N}$ and $\hat{\mathcal{H}}=\cup_{n\in \mathbb{N}}\hat{B}_{n}$.
Suppose that $\Phi$ is  a phase retrieval family of vectors in $\mathcal{H}$, then   we endow each $\hat{B}_{n}$ with the relative topology induced by $( \hat{\mathcal{H}}, \tau_{\Phi})$, denoted by $\tau_{\hat{B}_{n}}$ and
take  the direct limit topology $\tau_{\infty}$ on $\hat{\mathcal{H}}=\cup_{n\in \mathbb{N}} \hat{B}_{n}$. That means
$$ U\subset \hat{\mathcal{H}}  \textit { is open in  } \tau_{\infty}\textrm{  if and only if }  U\cap \hat{B}_{n} \textrm{  is open   in   } \tau_{\hat{B}_{n}} , \textrm {  for all  }  n\in \mathbb{N}. $$

An analogous approach to the proof of Theorem \ref{mainfinite} confirms that
    $(\hat{B}_{n}, \tau_{\hat{B}_{n}})$ is a compact  space, for all $n\in \mathbb{N}$.
In fact, assume that $w$ denotes the weak topology on Hilbert space $\mathcal{H}$. Then    the quotient  mapping $$\Lambda:(\mathcal{H},w)\rightarrow (\hat{\mathcal{H}}, \tau_{w}); \quad x \mapsto \hat{x}$$ is continuous.
Moreover, for each $n$, $B_{n}=\{x: \Vert x\Vert \leq n\}$ is a compact subset of $(\mathcal{H},w)$, by the Banach-Alaoglu theorem. And so, $\hat{B}_{n}:=\Lambda(B_{n})$ is compact in $(\hat{\mathcal{H}}, \tau_{w})$.


Again by using ,the continuity of
  the identity mapping $I:(\hat{B}_{n} ,\tau_{w})\rightarrow (\hat{B}_{n} , \tau_{\Phi})$,
the compactness of   $\hat{B}_{n}$  in $\tau_{w}$, for all $n\in \mathbb{N}$, and the fact that   $(\hat{\mathcal{H}}, \tau_{\Phi})$ is Hausdorff, we imply that  $\tau_{w}$ and $\tau_{\Phi}$ coincide on $\hat{B}_{n}$. Hence, $(\hat{B}_{n}, \tau_{\hat{B}_{n}})$  is   compact, for all $n\in \mathbb{N}$. Thus, we get a tower of compact spaces $\hat{B}_{1}\subseteq \hat{B}_{2}\subseteq ...\subseteq \hat{B}_{n}\subseteq...$. It follows from Lemma 3 of \cite{pentsak} that  $\hat{\mathcal{H}}$ is a paracompact space.
\end{proof}

The proof of Theorem \ref{paracompact}    leads to the following result and specially confirms that  the both topologies $\tau_{\Phi}$ and $\tau_{w}$  are equivalent  on the set $\hat{B}_{n}$, for any $n\in \mathbb{N}$.
\begin{corollary} \label{compact}
The following hold;
\begin{itemize}
\item[(i)] $(\hat{\mathcal{H}}, \tau_{\Phi})$ and $(\hat{\mathcal{H}}, \tau_{w})$  are  $\sigma$-compact.
\item[(ii)] Let $\Phi=\{\phi_{i}\}_{i\in I}$ be  a     bounded sequence of vectors   in     $\mathcal{H}$  which  does  phase retrieval.
Then the   topologies $\tau_{w}$ and $\tau_{\Phi}$ coincide on $\hat{B}_{n}$, for any $n\in \mathbb{N}$.
\end{itemize}
\end{corollary}
It is worth of note that,
$\hat{B}:=\{\hat{x}: \Vert x\Vert =1\}$ is not  compact in $(\hat{\mathcal{H}}, \tau_{\Phi})$ in general. In fact, the sequence $\{\hat{e}_{n}\}_{n=1}^{\infty}$ belongs to $\hat{B}$, while by taking $\Phi=\{e_{i}+e_{j}\}_{i<j}$ as in
    Example \ref{strict}, we get $\lim_{n}\vert \langle e_{n}, e_{i}+e_{j} \rangle \vert=0$, for every $i<j$, i.e., $\lim_{n}\hat{e}_{n}=0$    in $\tau_{\Phi}$.

\begin{remark}
Suppose $\Phi$ does phase retrieval in $\mathcal{H}$. Then,
as a comparison we get
\begin{eqnarray*}
\tau_{\Phi} \preceq  \tau_{w}  \preceq  \tau_{\infty}.
\end{eqnarray*}
Indeed, assuming $U\in \tau_{w}$ and applying Corollary \ref{compact} (ii), deduces that $U\cap \hat{B}_{n}\in \tau_{\hat{B}_{n}}$,  for all $n\in \mathbb{N}$. Hence, $U\in \tau_{\infty}$. This fact along with Lemma \ref{converges} (iii) complete the argument.
\end{remark}
 Moreover, take a collection of metric spaces $\{(X_{i},d_{i})\}_{i\in I}$, where $I$ is a totally ordered set, together with a collection of non-expansive maps $ f_{i,j}:X_{i}\rightarrow X_{j}$, for $ i \leq j$
 called the connecting maps
 with the property
$f_{i,k}=f_{i,j}\circ f_{j,k}$
 for all
$i\leq j\leq k$.
 Then,  considering the pairwise disjoint union of $X_{i}$ as in $X=\cup_{i\in I}X_{i}$, a pseudometric on  $X$ is defined by
$$d_{\infty} (x,y)=\inf\{d_{k}(f_{i,k}(x),f_{j,k}(y)); \quad i,j \leq k\}$$ for $x\in X_{i}$ and $y\in X_{j}$. It is said that $x,y\in X$ are equivalent, denoted $x\sim y$, if $d_{\infty} (x,y)=0$. The direct limit of the system of metric spaces $\{X_{i},f_{i,j}\}$  is the set $X/\sim$ together with the metric $d_{\infty}$  and is denoted by $\underrightarrow{lim}X_{i}$, \cite{Lund}.
Now, suppose $\Phi$ is a phase retrieval frame of $\mathcal{H}$. We observed that each  $(\hat{B}_{n}, \tau_{\hat{B}_{n}})$ is metrizable. Denote the metric associated with $ \tau_{\hat{B}_{n}}$ on $\hat{B}_{n}$ by $d_{n}$. Then
the direct limit metric on $\underrightarrow{lim}\hat{B}_{n}$ is defined by
\begin{eqnarray*}d_{\infty}(\hat{x},\hat{y})&=&\inf\{d_{k}(x,y): \quad \hat{x}\in \hat{B}_{i}, \hat{y}\in \hat{B}_{j}; i,j\leq k\}\\
&=& \inf\{d_{k}(x,y):  \hat{x}, \hat{y}\in \hat{B}_{k},\textrm{ for  } k\in \mathbb{N}  \}. \end{eqnarray*}
Note that in this case $\hat{\mathcal{H}}/\sim =\hat{\mathcal{H}}$ and thus $(\hat{\mathcal{H}},d_{\infty})$ is a metric space.
Furthermore, $$  \tau_{d_{\infty}} \preceq  \tau_{w} \preceq  \tau_{\infty}. $$
Indeed, assume that $\{\hat{x}_{k}\}_{k}$ converges to $\hat{x}$ in $\tau_{w}$, this implies that there exists $N$ such that  $\hat{x}_{k},\hat{x}\in \hat{B}_{N}$, for every $k$. Thus,
 $\lim_{k}d_{N}(\hat{x}_{k},\hat{x})=0$ and so
$$\lim_{k}d_{\infty}(\hat{x}_{k},\hat{x})=\lim_{k}\inf\{d_{n}(\hat{x}_{k},\hat{x}): n\in \mathbb{N}, \hat{x}_{k},\hat{x}\in \hat{B}_{n}\}\leq \lim_{k}d_{N}(\hat{x}_{k},\hat{x})=0.$$

\subsection{Topometric spaces}
In the sequel, we present a brief discussion  on topometric spaces, namely spaces equipped both with a topology and with a metric, and show that the triple $(\hat{\mathcal{H}}, \tau_{\Phi},d_{\Phi})$ constitutes a topometric space. This    new perspective provides a fresh outlook on the properties of   $\hat{\mathcal{H}}$.
\begin{definition}\label{20}\cite{yaacov}
A (Hausdorff) topometric space is a triple $(X, \tau ,d)$ where $X$ is a set of points, $\tau$ is a topology on $X$ and $d$ is a metric on $X$:
\begin{itemize}
\item[(i)] The metric refines the topology. In the other words, for every open set $U\subseteq X$ and every $x\in U$ there exists $r>0$ such that $N_{r}(x)\subseteq U$.
\item[(ii)] The metric function $d:X^{2} \rightarrow [0,\infty]$ is lower semi-continuous, i.e., for all $r>0$ the set $\{(x,y)\in X^{2} : d(x,y)\leq r\}$ is closed in $X^{2}$.
\end{itemize}
\end{definition}
 Also, a family of functions $\mathcal{F}\subseteq \mathbb{C}^{X}$ (where $\mathbb{C}^{X}$ denotes the set of all functions from $X$ into $\mathbb{C}$) is said to separate the points from  closed sets if for every closed set $\mathcal{F}\subseteq X$and $x\in   \mathcal{F}^{c}$ there exists a function $f\in \mathcal{F}$ which is constant on $\mathcal{F}$ and takes some different value at $x$.
\begin{definition}\label{211}
Let    $(X, \tau ,d)$  be a topometric space. Say that a family of functions $\mathcal{F}\subseteq C_{L(l)}(X)$ ($1$-Lipschitz functions from $X$ to $\mathbb{C}$) is sufficient if
\begin{itemize}
\item[(i)] It separates points and closed sets.
\item[(ii)] For $x,y\in X$ we have $d(x,y)=\sup \{ \left\vert f(x)-f(y)\right\vert : f\in \mathcal{F}\}$.
\end{itemize}
\end{definition}
A topometric space $X$ is completely regular if $C_{L(l)}(X)$  is sufficient.
Now, assume that $\Phi$ is a phase retrieval sequence in $\mathcal{H}$. We consider the  triple $(\hat{\mathcal{H}}, \tau_{\Phi},d_{\Phi})$, and we note that by applying Lemma  \ref{finer} the metric $d_{\Phi}$ refines the topology $\tau_{\Phi}$.  The aim is  to establish that $(\hat{\mathcal{H}}, \tau_{\Phi}, d_{\Phi})$ forms a topometric space, thereby implying its complete regularity.

\begin{proposition}
 The  triple $(\hat{\mathcal{H}}, \tau_{\Phi},d_{\Phi})$ constitutes  a topometric space.
\end{proposition}
\begin{proof}
Using  Lemma \ref{finer} the metric $d_{\Phi}$ is finer than  $\tau_{\Phi}$ and so the metric refines the topology.
 It is enough to show that
$$ O=\{(\hat{x},\hat{y})\in \hat{\mathcal{H}}\times \hat{\mathcal{H}} : d_{\Phi}(\hat{x},\hat{y})\leq  r\}$$
is closed, for each $r>0$. Suppose that $(\hat{a},\hat{b})\in \overline{O}$, then there exists a sequence $\{(\hat{x}_{n},\hat{y}_{n})\}_{n=1}^{\infty}\subseteq \hat{\mathcal{H}}\times \hat{\mathcal{H}}$ such that
$\lim_{n}d_{\Phi}(\hat{x}_{n},\hat{a})=0$ and  $\lim_{n}d_{\Phi}(\hat{y}_{n},\hat{b})=0$. Thus, for every $\epsilon>0$ there exists $N\in \mathbb{N}$ such that  $\max\{d_{\Phi}(\hat{x}_{N},\hat{a}), d_{\Phi}(\hat{y}_{N},\hat{b})\}<\epsilon$. Hence,
\begin{eqnarray*}
d_{\phi}(\hat{a},\hat{b})&\leq& d_{\phi}(\hat{a},\hat{x}_{N})+d_{\phi}(\hat{x}_{N},\hat{y}_{N})+d_{\phi}(\hat{y}_{N},\hat{b})\\
&<& \epsilon+r+\epsilon=2\epsilon+r.
\end{eqnarray*}
Consequently, $d_{\phi}(\hat{a},\hat{b})\leq r$ and so $(\hat{a},\hat{b})\in O$.
\end{proof}
In view of Definition \ref{211} we note that the family $\{\rho_{\phi_{j}}\}_{j\in I}$ satisfies in condition $(ii)$. In addition,
since $\Phi$ yields phase retrieval, by applying Proposition 2.4.8 of \cite{book1}, this family separates points and closed sets. Hence, $(\hat{\mathcal{H}}, \tau_{\Phi},d_{\Phi})$ is a completely regular topometric space.

 Furthermore,  we examine the Lipschitz continuity of $\alpha_{\Phi}$ on  $\hat{\mathcal{H}}$  with respect to the metric $d_\Phi$.
 \begin{proposition}
The non-linear mapping $\alpha_\Phi: (\hat{\mathcal{H}}_n, d_\Phi) \to \ell^2(I_m)$ is bi-Lipschitz with respect to the metric $d_\Phi$. via  the lower and upper Lipschitz bounds $1$ and $\sqrt{m}$, respectively.
\end{proposition}
\begin{proof}
  Using the definition of the metric $d_\Phi$ associated with the phase retrieval frame $\Phi = \{\phi_j\}_{j \in I_m}$, we obtain:
\begin{eqnarray*}
d^{2}_{\phi}(\hat{x},\hat{y})&=&\max_{j\in I_{m}} \left\vert\vert\langle x,\phi_{j}\rangle\vert-\vert\langle y,\phi_{j} \rangle\vert\right\vert^{2}\\
&\leq& \sum_{j\in I_{m}}\left\vert\vert\langle x,\phi_{j}\rangle\vert-\vert\langle y,\phi_{j}\rangle\vert\right\vert^{2}\\
&=& \Vert \alpha_{\Phi}\hat{x}-\alpha_{\Phi}\hat{y}\Vert^{2}\\
&\leq& m d^{2}_{\phi}(\hat{x},\hat{y}).
\end{eqnarray*}
This  completes the proof.
\end{proof}



\subsection*{Availability of data and materials}
Data sharing not applicable to this article as no datasets were
generated during the preparation of this paper.

\subsection*{Funding} The authors received no financial support for the research,
authorship, and/or publication of this article.

\subsection*{Conflict of interest}
This work does not have any conflicts of interest.

\subsection*{Author Contributions}
All authors discussed the results and contributed to the final
manuscript.

\bibliographystyle{amsplain}

\end{document}